\documentclass[10pt,reqno]{amsart}
\usepackage{amssymb,amscd,graphicx,stmaryrd,mathrsfs}
\usepackage[all]{xy}
\usepackage{color}
\usepackage{url}
\newtheorem{theorem}{Theorem}
\newtheorem{proposition}[theorem]{Proposition}

\newtheorem{lemma}[theorem]{Lemma}
\newtheorem{remark}[theorem]{Remark}

\begin{document}

\title[Pusz--Woronowicz's functional calculus]{Pusz--Woronowicz's functional calculus revisited}
\author[K.\ Hatano]{Kanae Hatano}
\author[Y.\ Ueda]{Yoshimichi UEDA$^1$}
\address{
Graduate School of Mathematics, Nagoya University, 
Furocho, Chikusaku, Nagoya, 464-8602, Japan
}
\email{
(KH) m19033f@math.nagoya-u.ac.jp; 
(YU) ueda@math.nagoya-u.ac.jp
}
\date{\today}
\thanks{$^1$Supported by Grant-in-Aid for Scientific Research (B) JP18H01122} 
\maketitle
\begin{abstract}
This note is a complement to Pusz--Woronowicz's works on functional calculus for two positive forms from the viewpoint of operator theory. Based on an elementary, self-contained and purely Hilbert space operator explanation of their functional calculus, we show that any operator connection type operations (including any operator perspectives) are captured by their functional calculus.
\end{abstract}

\allowdisplaybreaks{

\section{Introduction}

It is well known that the notion of geometric mean for Hilbert space operators originates in Pusz--Woronowicz's work \cite{PuszWoronowicz:RMP75}. However, it seems that Pusz--Woronowicz's approach (that precisely means their construction and/or formulation) has not received much attention apart from some works in the context of mathematical physics based on operator algebras such as   \cite{Uhlmann:CMP77,Kosaki:CMP82,Kosaki:JOT86,Donald:CMP86,Yamagami:LMP08,Yamagami:LMP19}. In fact, almost all the papers discussing operator means seem to follow Ando's translation \cite[Theorem I.2]{Ando:LectNotes78} of Pusz--Woronowicz's geometric mean into Hilbert space operators, which is, strictly speaking, not the same as Pusz--Woronowicz's original construction. This circumstance motivated us to revisit Pusz--Woronowicz's original papers \cite{PuszWoronowicz:RMP75,PuszWoronowicz:LMP78} in the context of operator theory. 

We originally thought, by a remark in \cite{Ando:LectNotes78}, that Ando made a kind of direct calculation (rather than any characterization like \cite[Theorem I.2]{Ando:LectNotes78}) to translate Pusz--Woronowicz's geometric mean for positive forms into the language of Hilbert space operators. However, he did not explain anything about such a calculation. Hence, we tried to rediscover such a calculation. Although such an attempt is standard to learn new subjects, the calculation we observed shows that any operator perspective (see \cite{Effros:NASUSA09,EbadianNikoufarGordji:PNAS11,EffrosHansen:AFA14}) is a special case of Pusz--Woronowicz's functional calculus. This has unexpectedly been unnoticed so far.  

The main purpose of this note is to clarify the role and the merit of Pusz--Woronowicz's functional calculus in operator theory. Hence, we will also \emph{reconstruct their functional calculus in the framework of Hilbert space bounded operators} from their original works (see \S\S3.1). After our initial observation was made, we were aware of J.I.\ Fujii's note \cite{Fujii:MathJapon88} by thorough search. (See also his earlier related work \cite{Fujii:MathJapon80}, which we learned from him.) His discussion heavily depends upon the structure theorem \cite[Theorem 3.4]{KuboAndo:MathAnn80} for operator means. On the other hand, ours is straightforward, constructive and hence completely general beyond operator means. Namely, one of the consequences of this note says that \emph{any} operator connection type operations $(A,B)\mapsto A^{1/2}f(A^{-1/2}BA^{-1/2})A^{1/2}$ and $B^{1/2}g(B^{-1/2}AB^{-1/2})B^{1/2}$ with locally bounded Borel $f(t),g(t)$ are exactly Pusz--Woronowicz's functional calculus with functions $(r,s)\mapsto rf(s/r)$ and $(r,s)\mapsto sg(r/s)$, respectively. See Theorem \ref{T4}. 

An obvious merit of Pusz--Woronowicz's functional calculus is that its definition does not involve the procedure of taking the limit of $(A+\varepsilon I, B+\varepsilon I)\to(A,B)$ as $\varepsilon\searrow0$. Another merit is that the joint operator convexity question for their functional calculus can easily be investigated (see \S\S4.1). Moreover, their functional calculus seems to fit unbounded functions like $(r,s) \mapsto r\log(r/s)$. On the other hand, the formulation of Kubo--Ando's operator connections seems more manageable with respect to the order structure (thanks to many existing works). Thus, there are two kinds of representations of the same binary operation for Hilbert space operators, each of which has own merits.  

This note grew out from the first author's master thesis project under the second author's guidance. In her master thesis, all the materials were treated in the finite-dimensional setting by using the specific feature of finite-dimensionality. 

\section{Preliminaries} 

\subsection{Notations and terminologies} The inner product $(\,\cdot\,|\,\cdot\,)_\mathcal{H}$ of a Hilbert space $\mathcal{H}$; the spectrum $\sigma(T)$ of a Hilbert space operator $T$; all the bounded operators $B(\mathcal{H})$ on a Hilbert space $\mathcal{H}$; all the continuous functions $C(\Omega)$ on a locally compact Hausdorff space $\Omega$; $\mathscr{B}(\Omega)$ and $\mathscr{B}_\mathrm{b}(\Omega)$, $\mathscr{B}_\mathrm{locb}(\Omega)$ denote, respectively, all the Borel functions, all the bounded Borel functions and all the Borel functions that are bounded on any compact subsets, on a locally compact Hausdorff space $\Omega$. Remark that $C(\Omega) \subset \mathscr{B}_\mathrm{locb}(\Omega)$ and moreover, $C(\Omega) \subset \mathscr{B}_\mathrm{b}(\Omega) = \mathscr{B}_\mathrm{locb}(\Omega)$ when $\Omega$ is compact. Throughout this note, an operator is said to be \emph{invertible}, if it has a bounded inverse. 

\subsection{Functional calculus for two commuting selfadjoint operators}
For the reader's convenience, we briefly review the functional calculus for two commuting selfadjoint bounded operators $S, T$ on a Hilbert space $\mathcal{H}$, which seems to be less popular. 

Consider $N := S+iT$, a normal bounded operator on $\mathcal{H}$. By the spectral mapping theorem we have $\sigma(S) = \{\mathrm{Re}z\,;\,z\in \sigma(N)\}$ and $\sigma(T) = \{\mathrm{Im}z\,;\,z\in \sigma(N)\}$, and hence $\sigma(N) \subset \sigma(S)+i\sigma(T)$. Let $f \in \mathscr{B}_\mathrm{b}(\sigma(S)\times\sigma(T))$ be arbitrarily given. Then, $\hat{f}(z) := f(\mathrm{Re}z,\mathrm{Im}z)$ falls into $\mathscr{B}_\mathrm{b}(\sigma(N))$, and we define $\Phi_{(S,T)}(f) := \hat{f}(N)$, the functional calculus of $N$ with $\hat{f}$ (see \cite[\S4.5]{Pedersen:GTM}). By construction, $\Phi_{(S,T)} : \mathscr{B}_\mathrm{b}(\sigma(S)\times\sigma(T)) \to B(\mathcal{H})$ is obviously a unital $*$-homomorphism sending two coordinate functions $(s,t) \mapsto s,t$ to $S,T$, respectively, such that it possesses the following monotone convergence property: If $\{f_n\}$ is a bounded non-decreasing sequence in $\mathscr{B}_\mathrm{b}(\sigma(S)\times\sigma(T))$, then $f = \sup_n f_n \in \mathscr{B}_\mathrm{b}(\sigma(S)\times\sigma(T))$ and $\Phi_{(S,T)}(f_n) \nearrow \Phi_{(S,T)}(f)$ (in the strong operator topology) as $n\to\infty$. 

Let $\Omega$ be a compact subset of $\mathbb{R}^2$ that contains $\sigma(S)\times\sigma(T)$, and let $\Psi : \mathscr{B}_\mathrm{b}(\Omega) \to B(\mathcal{H})$ be another unital $*$-homomorphism with the same properties as $\Phi_{(S,T)}$ has. Here, $f$ is understood as its restriction to $\sigma(S)\times\sigma(T)$ when $\Phi_{(S,T)}(f)$ is considered. Observe that $\Psi(p) = \Phi_{(S,T)}(p)$ holds for any polynomial $p=p(s,t)$. Taking a uniform approximation by polynomials on $\Omega$, we see that $\Psi(f) = \Phi_{(S,T)}(f)$ holds for any $f \in C(\Omega)$. Let $\mathscr{B}$ be all the $f \in \mathscr{B}_\mathrm{b}(\Omega)$ with $\Psi(f) = \Phi_{(S,T)}(f)$. We have seen that $C(\Omega)$ sits in $\mathscr{B}$. Moreover, $\mathscr{B}$ is clearly closed under the monotone convergence. Therefore, the monotone class theorem ensures that $\mathscr{B} = \mathscr{B}_\mathrm{b}(\Omega)$; namely, $\Psi = \Phi_{(S,T)}$ holds. Consequently, we write $f(S,T) := \Phi_{(S,T)}(f)$ with $f \in \mathscr{B}_\mathrm{b}(\Omega)$ and call $f \mapsto f(S,T)$ \emph{the bounded Borel functional calculus of the commuting selfadjoint operators $S,T$}. By the above discussion, this functional calculus map clearly factors through $\mathscr{B}_\mathrm{b}(\sigma(S)\times\sigma(T))$. This is analogous to the single operator functional calculus. 

Remark that any $g \in \mathscr{B}_\mathrm{b}(\sigma(S))$ can be regarded as an element of $\mathscr{B}_\mathrm{b}(\sigma(S)\times\sigma(T))$ with $g(s,t):=g(s)$. Thus we can consider $g(S,T)$ for that function $g$. Trivially, the mapping $g \in \mathscr{B}_\mathrm{b}(\sigma(S)) \mapsto g(S,T) \in B(\mathcal{H})$ is a unital $*$-homomorphism sending the coordinate function $s$ to $S$ such that it satisfies the monotone convergence property. By the uniqueness of functional calculus we obtain that $g(S,T) = g(S)$ holds for any $g(s,t)=g(s)$ with $g \in \mathscr{B}_\mathrm{b}(\sigma(S))$. Similarly, we have $h(S,T) = h(T)$ holds for any $h(s,t)=h(t)$ with $h \in \mathscr{B}_\mathrm{b}(\sigma(T))$. 

It is more tractable to start with the joint spectral measure in the finite-dimensional setting. Namely, let $P_\lambda$ ($\lambda \in \sigma(S)$) and $Q_\mu$ ($\mu\in\sigma(T)$) be the spectral projections associated with $S$ and $T$, respectively. Then, the joint spectral projection $E_{(\lambda,\mu)}$ at $(\lambda,\mu)\in\sigma(S)\times\sigma(T)$ is simply given by $E_{(\lambda,\mu)}=P_\lambda Q_\mu = Q_\mu P_\lambda$, and the functional calculus is constructed by $f(S,T)=\sum_{(\lambda,\mu)\in\sigma(S)\times\sigma(T)} f(\lambda,\mu)\,E_{(\lambda,\mu)}$. 

\section{Pusz--Woronowicz's functional calculus for operators}

Let $A, B$ be positive bounded operators on a Hilbert space $\mathcal{H}$ \emph{throughout this section}.  

\subsection{Construction} 
Pusz--Woronowicz's functional calculus is defined for positive forms on complex vector spaces, and it can be  applied to Hilbert space operators (see at the end of this subsection). However, we prefer to reconstruct their functional calculus in terms of Hilbert space operators explicitly, because the original formulation is too abstract to discuss Hilbert space operators.  

\medskip
We first consider the orthogonal decomposition $\mathcal{H}=\mathcal{H}_0\oplus\mathcal{H}_1$ with $\mathcal{H}_0 := \mathrm{Ker}(A+B)$ and $\mathcal{H}_1:=\mathcal{H}_0^\perp$. Since $0 \leq A, B \leq A+B$, we can write 
\[
A = \begin{bmatrix} 0 & 0 \\ 0 & A_1 \end{bmatrix}, \quad 
B = \begin{bmatrix} 0 & 0 \\ 0 & B_1 \end{bmatrix}
\]
along the above orthogonal decomposition of $\mathcal{H}$. By construction, one has $\mathrm{Ker}(A_1+B_1)=\{0\}$. 

The mappings $(A_1+B_1)^{1/2}x \mapsto A_1^{1/2}x, (A_1+B_1)^{1/2}x \mapsto B_1^{1/2}x$ ($x \in \mathcal{H}_1$) uniquely extend to contractive operators $X, Y$ on $\mathcal{H}_1$ so that $X(A_1+B_1)^{1/2} = A_1^{1/2}, Y(A_1+B_1)^{1/2} =
B_1^{1/2}$ hold. We set $R:= |X|^2, S:= |Y|^2$. Here is a simple lemma, which is a key of Pusz--Woronowicz's functional calculus. 

\begin{lemma}\label{L1} {\rm(\cite[Theorem 1.1]{PuszWoronowicz:RMP75})} $R+S=I_{\mathcal{H}_1}$. \end{lemma}
\begin{proof}
The desired assertion follows from $((R+S)(A_1+B_1)^{1/2}x\,|\,(A_1+B_1)^{1/2}y)_{\mathcal{H}_1} = (A_1^{1/2}x\,|A_1^{1/2}y)_{\mathcal{H}_1} + (B_1^{1/2}x\,|B_1^{1/2}y)_{\mathcal{H}_1} = ((A_1+B_1)^{1/2}x\,|\,(A_1+B_1)^{1/2}y)_{\mathcal{H}_1}$, since $(A_1+B_1)^{1/2}$ has dense range.
\end{proof}

A function $f(r,s)$ on $[0,\infty)^2$ (or $(0,\infty)^2$) is said to be \emph{homogeneous} if $f(\lambda r,\lambda s) = \lambda f(r,s)$ for all $\lambda,r,s \geq 0$ (resp.\ $\lambda,r,s>0$). 
We will construct an operator, say $f(A,B)$, for a homogeneous $f \in \mathscr{B}_\mathrm{locb}([0,\infty)^2)$, and call it \emph{Pusz--Woronowicz's functional calculus} for the pair $(A,B)$ with $f$ (see the final paragraph of this subsection for the reason). The local boundedness of $f$ is the requirement to guarantee that the resulting $f(A,B)$ becomes bounded. Remark that the class $\mathscr{B}_\mathrm{locb}([0,\infty)^2)$ is too restrictive for some applications; see the discussion in \S\S4.2. 

Since $R,S$ are commuting selfadjoint operators, we can perform the functional calculus $f(R,S)$; see \S2 ({\it n.b.}, $f(r,s)$ is bounded on $\sigma(R)\times\sigma(S)$). Then we set 
\begin{equation}\label{Eq1}
f(A,B) := \begin{bmatrix} 0 & (A_1+B_1)^{1/2} \end{bmatrix}^* f(R,S) \begin{bmatrix} 0 & (A_1+B_1)^{1/2} \end{bmatrix}
\end{equation}
along $\mathcal{H}=\mathcal{H}_0\oplus\mathcal{H}_1$, where the $1\times2$ operator matrix above is nothing but $(A+B)^{1/2}$ but we regard it as an operator from $\mathcal{H}$ to $\mathcal{H}_1$ with dense range. We should remark that formula \eqref{Eq1} has a more natural expression. Namely, we have
\begin{equation}\label{Eq2}
f(A,B) = \begin{bmatrix} 0 & 0 \\ 0 & (A_1+B_1)^{1/2}f(R,S)(A_1+B_1)^{1/2} \end{bmatrix}
\end{equation}
along $\mathcal{H} = \mathcal{H}_0\oplus\mathcal{H}_1$. Moreover, since $f(0,0) = 0$ by homogeneity, the following description of $f(A,B)$ is also justified: 
\begin{equation}\label{Eq3}
f(A,B) = (A+B)^{1/2} f(R,S) (A+B)^{1/2}, 
\end{equation} 
where we regard $R, S$ as operators on the whole $\mathcal{H}$ by letting $R=S=0$ on $\mathcal{H}_0$. We will use these \eqref{Eq1}-\eqref{Eq3} as the definition of $f(A,B)$.    

\medskip
When both $A,B$ are invertible (or more generally, $A+B$ is invertible), one easily observes that $\mathcal{H}_1=\mathcal{H}$ (hence $\mathcal{H}_0 = \{0\}$) and 
\begin{equation}\label{Eq4}
R = (A+B)^{-1/2}A(A+B)^{-1/2}, \quad S=(A+B)^{-1/2}B(A+B)^{-1/2},
\end{equation}
which are clearly invertible too. (This fact was already remarked in \cite[Example 3]{Uhlmann:CMP77} but not discussed seriously there.) Hence the construction of $f(A,B)$ so far works well for any homogeneous $f \in \mathscr{B}_\mathrm{locb}((0,\infty)^2)$ rather than $\mathscr{B}_\mathrm{locb}([0,\infty)^2)$, when both $A, B$ are invertible.

\medskip
Here we briefly explain the relationship between the above construction and the original one in \cite{PuszWoronowicz:RMP75}. For two positive forms $\alpha, \beta$ on a complex vector space and a homogeneous $f \in \mathscr{B}_\mathrm{locb}([0,\infty)^2)$, Pusz--Woronowicz's original functional calculus defines a new positive form $f(\alpha,\beta)$ on the same vector space. Consider two positive forms $\alpha(x,y):=(Ax\,|\,y)_\mathcal{H}, \beta(x,y):=(Bx\,|\,y)_\mathcal{H}$, $x,y \in \mathcal{H}$. The original construction starts with a compatible representation of $\alpha,\beta$. The pair $(R,S)$ constructed above together with the map $h : \mathcal{H} \to \mathcal{H}_1$ defined by $h(x):=(A+B)^{1/2}x$, $x\in\mathcal{H}$, is indeed a compatible representation; namely, $R+S=I_{\mathcal{H}_1}$ and
\[
(R h(x)\,|\,h(y))_{\mathcal{H}_1} = (Ax\,|\,y)_\mathcal{H}, \quad 
(S h(x)\,|\,h(y))_{\mathcal{H}_1} = (Bx\,|\,y)_\mathcal{H}, \qquad x,y \in \mathcal{H} 
\]
hold. Then, the desired positive form $f(\alpha,\beta)$ is defined to be $(x,y) \in \mathcal{H}\times\mathcal{H} \mapsto (f(R,S)h(x)\,|\,h(y))_{\mathcal{H}_1} \in \mathbb{C}$ based on the chosen compatible representation $(R,S,h)$. Obviously, $f(\alpha,\beta)(x,y) = (f(A,B)x\,|\,y)_\mathcal{H}$ holds for any $x,y \in \mathcal{H}$ in the case. Here, we point out that Pusz--Woronowicz proved, in \cite[Theorem 1.2]{PuszWoronowicz:RMP75}, that the resulting form $f(\alpha,\beta)$ does not depend on the choice of compatible representation, where the homogeneity of function plays a crucial role. We recommend the reader to consult Pusz--Woronowicz's original paper \cite{PuszWoronowicz:RMP75} that contains many fruitful ideas. 

\subsection{Properties}
We believe that the facts given in this subsection are new. 

\medskip
The next proposition justifies the notation of $f(A,B)$. Although Pusz--Woronowicz did not give it, we believe that they confirmed it before the release of their works. Remark that the construction of $f(A,B)$ so far works with function $f(r,s)$ defined only on the segment $r+s=1$ with $r,s \geq 0$. However, the next proposition needs that $f(r,s)$ is homogeneous and hence defined on the whole $[0,\infty)^2$ (or $(0,\infty)^2$ when both $A,B$ are invertible). 

\begin{proposition}\label{P2} If $A, B$ commute with each other, then the above $f(A,B)$ is just the usual functional calculus of $A,B$ {\rm(}see \S2{\rm)}. 
\end{proposition} 
\begin{proof} We will give a proof when $f \in \mathscr{B}_\mathrm{locb}([0,\infty)^2)$. The case when $A,B$ are invertible and $f\in\mathscr{B}_\mathrm{locb}((0,\infty)^2)$ can be treated with the same method. 

Notice that $A_1, B_1$ commute with each other too. By formula \eqref{Eq2} and $f(0,0)=0$ due to homogeneity, it suffices to prove that $(A_1+B_1)^{1/2} f(R,S)(A_1+B_1)^{1/2}$ is exactly the functional calculus $f(A_1,B_1)$. 

We first remark that $g(A_1+B_1)=\hat{g}(A_1,B_1)$ with $\hat{g}(r,s) := g(r+s)$ holds for any $g \in \mathscr{B}_\mathrm{b}[0,m]$ with $m:=\Vert A_1+B_1\Vert$. Clearly, $g \in \mathscr{B}_\mathrm{b}[0,m] \mapsto \hat{g} \in \mathscr{B}_\mathrm{b}(\sigma(A)\times\sigma(B))$ is a unital $*$-homomorphism with the sequential pointwise convergence property, that is, $g_n \to g$ pointwise implies $\hat{g}_n \to \hat{g}$ pointwise. Thus, $g \in \mathscr{B}_\mathrm{b}[0,m] \mapsto \hat{g}(A_1,B_1) \in B(\mathcal{H}_1)$ defines a unital $*$-homomorphism with the monotone convergence property such that $g(t)=t$ implies $g(A_1+B_1)=A_1+B_1$. The uniqueness of functional calculus shows the desired remark. In particular, $(A_1+B_1)^{1/2}$ can also be understood as the functional calculus of pair $(A_1,B_1)$ with function $(r,s) \mapsto \sqrt{r+s}$.  

Denote by $(a,b)$ the natural coordinates in $[0,\infty)^2$. Consider the functions
\[
x(a,b) := 
\begin{cases} 
\sqrt{a}/\sqrt{a+b} & (a+b\neq 0), \\
0  & (a+b=0);  
\end{cases} \qquad 
y(a,b) := 
\begin{cases} 
\sqrt{b}/\sqrt{a+b} & (a+b\neq 0), \\
0  & (a+b=0), 
\end{cases}
\]
both of which define bounded Borel functions on $\sigma(A_1)\times\sigma(B_1)$. Since $x(a,b)\sqrt{a+b} = \sqrt{a}$ and $y(a,b)\sqrt{a+b} = \sqrt{b}$ hold for any $a,b\geq0$, we have $x(A_1,B_1)(A_1+B_1)^{1/2}=A_1^{1/2}$ and $y(A_1,B_1)(A_1+B_1)^{1/2}=B_1^{1/2}$. Consequently, $X = x(A_1,B_1)$ and $Y=y(A_1,B_1)$ hold by the uniqueness of $X, Y$. Then, $r(a,b) := x(a,b)^2$ and $s(a,b) := y(a,b)^2$ define bounded Borel functions on $\sigma(A_1)\times\sigma(B_1)$; hence $R = |X|^2 = r(A_1,B_1)$ and $S = |Y|^2 = s(A_1,B_1)$.   

For any $h \in \mathscr{B}_\mathrm{b}([0,1]^2)$, $\hat{h}(a,b) := h(r(a,b),s(a,b))$, $(a,b) \in \sigma(A_1)\times\sigma(B_1)$, defines an element of $\mathscr{B}_\mathrm{b}(\sigma(A_1)\times\sigma(B_1))$, and $h \in \mathscr{B}_\mathrm{b}([0,1]^2) \mapsto \hat{h} \in \mathscr{B}_\mathrm{b}(\sigma(A_1)\times\sigma(B_1))$ is a unital $*$-homomorphism with the sequential pointwise convergence property. Thus, we have a unital $*$-homomorphism $h \in \mathscr{B}_\mathrm{b}([0,1]^2) \mapsto \hat{h}(A_1,B_1) \in B(\mathcal{H}_1)$ with the monotone convergence property. If $h(r,s)=r$ in $(r,s) \in [0,1]^2$, then $\hat{h}(A_1,B_1) = r(A_1,B_1) = R$. Similarly, if $h(r,s) = s$, then $\hat{h}(A_1,B_1) = S$. By the uniqueness of functional calculus (see \S2), $\hat{h}(A_1,B_1) = h(R,S)$ holds for every $h \in \mathscr{B}_\mathrm{b}([0,1]^2)$.

Let $f \in \mathscr{B}_\mathrm{locb}([0,\infty)^2)$ be homogeneous. We will understand $f$ as an element of $\mathscr{B}_\mathrm{b}([0,1]^2)$ as well as one of $\mathscr{B}_b(\sigma(A)\times\sigma(B))$. Observe, by homogeneity, that 
\begin{align*} 
\sqrt{a+b}\,\hat{f}(a,b)\,\sqrt{a+b} 
&= 
(a+b)f(a/(a+b),b/(a+b)) 
= f(a,b)
\end{align*}
if $a+b>0$. Moreover, the same identity trivially holds even when $a+b=0$. Consequently, 
\[
(A_1+B_1)^{1/2}f(R,S)(A_1+B_1)^{1/2} = (A_1+B_1)^{1/2}\hat{f}(A_1,B_1)(A_1+B_1)^{1/2}  = f(A_1,B_1) 
\]
as desired.  
\end{proof} 

Using joint spectral measures one can give a more intuitive proof to the above fact in the finite-dimensional setting. 

\medskip
We will then express $f(A,B)$ like operator connections in the sense of Kubo--Ando \cite{KuboAndo:MathAnn80} in a constructive way. We start with the next simple (but key) lemma. Although the lemma is quite easy to prove, we do give its proof. 

\begin{lemma}\label{L3} Let $S$ be an invertible bounded operator on a Hilbert space $\mathcal{K}$ and $T$ be a self-adjoint bounded operator on $\mathcal{K}$. If $S^{-1}TS$ is still selfadjoint, then $g(S^{-1}TS) = S^{-1}g(T)S$ holds for any $g \in \mathscr{B}_\mathrm{locb}(\mathbb{R})$. 
\end{lemma} 
\begin{proof}
For any polynomial $p(x)$ we observe that $p(S^{-1}TS) = S^{-1}p(T)S$. For a continuous function $g(x)$ on $[-m,m]$ with a sufficiently large $m > 0$, we take a uniform approximation to $f(x)$ by polynomials on $[-m,m]$ and obtain $g(S^{-1}TS) = S^{-1}g(T)S$.  

Let $\mathscr{B}$ be all the $h \in \mathscr{B}_\mathrm{b}[-m,m]$ such that $h(S^{-1}TS) = S^{-1}h(T)S$. By means of bounded Borel functional calculus one can easily see that $\mathscr{B}$ is closed under the monotone convergence. Moreover, the above observation shows that $\mathscr{B}$ contains all the continuous functions. Hence $\mathscr{B}$ is exactly $\mathscr{B}_\mathrm{b}[-m,m]$ by the monotone class theorem, and the desired assertion follows.   
\end{proof}

Here is the main observation. In what follows, the limit $(\varepsilon_1,\varepsilon_2) \searrow (0,0)$ should be understood as $\varepsilon_1 + \varepsilon_2 \searrow 0$. Also, $f(T,1), f(1,T)$ denote the functional calculus of $T$ with functions $t\mapsto f(t,1), f(1,t)$, respectively. 

\begin{theorem}\label{T4} 
The following hold true{\rm:} 
\begin{itemize}
\item[(1)] For any homogeneous $f \in \mathscr{B}_\mathrm{locb}([0,\infty)^2)$, 
\[
f(A,B) = 
\begin{cases} 
A^{1/2} f(1,A^{-1/2}B A^{-1/2})A^{1/2} &\text{if $A$ is invertible}, \\
B^{1/2} f(B^{-1/2}A B^{-1/2},1)B^{1/2} &\text{if $B$ is invertible}. 
\end{cases}
\]
When both $A,B$ are invertible,  
\[
f(A,B) = A^{1/2}f(1,A^{-1/2}B A^{-1/2})A^{1/2} = B^{1/2} f(B^{-1/2}A B^{-1/2},1)B^{1/2} 
\]
holds for any homogeneous $f \in \mathscr{B}_\mathrm{locb}((0,\infty)^2)$. 
\item[(2)]
If $0$ is an isolated point in $\sigma(A+B)$, then for any homogeneous $f \in C([0,\infty)^2)$, 
\begin{align*}
f(A,B)
= 
\lim_{(\varepsilon_1,\varepsilon_2)\searrow(0,0)}f(A_{\varepsilon_1},B_{\varepsilon_2}) 
\end{align*}
in the operator norm topology, where $A_{\varepsilon_1} := A+\varepsilon_1 I_\mathcal{H}, B_{\varepsilon_2} := B+\varepsilon_2 I_\mathcal{H}$. In particular, this always holds for any pair of positive operators on a finite dimensional Hilbert space. 
\end{itemize}
\end{theorem}
\begin{proof} 
Item (1): 
Since $f(r,s)$ is homogeneous, we have $f(r,s)=rf(1,r^{-1}s)=f(rs^{-1},1)s$ for all $r,s>0$, and hence $f(R,S) = Rf(1,R^{-1}S)=f(RS^{-1},1)S$, where $f(1,R^{-1}S), f(RS^{-1},1)$ can also be understood as the functional calculus of $R^{-1}S$, $RS^{-1}$ with functions $t \mapsto f(1,t), f(t,1)$, respectively, thanks to the uniqueness of functional calculus as in the proof of Proposition \ref{P2}.  

Assume that $A$ is invertible. Observe that
\[
R^{-1}S = (A^{1/2}(A+B)^{-1/2})^{-1}(A^{-1/2}BA^{-1/2})(A^{1/2}(A+B)^{-1/2})
\] 
is selfadjoint, since $R,S$ commute with each other. 
It follows by Lemma \ref{L3} that
\begin{align*}
&f(R,S) = Rf(1,R^{-1}S) \\
&= 
(A+B)^{-1/2}A(A+B)^{-1/2} (A^{1/2}(A+B)^{-1/2})^{-1} f(1,A^{-1/2}BA^{-1/2})(A^{-1/2}(A+B)^{1/2}) \\
&=
(A+B)^{-1/2}A^{1/2}f(1,A^{-1/2}BA^{-1/2})A^{1/2}(A+B)^{-1/2}. 
\end{align*}
Similarly, if $B$ is invertible, then $RS^{-1} = (B^{-1/2}(A+B)^{1/2})^{-1}(B^{-1/2}AB^{-1/2})(B^{-1/2}(A+B)^{1/2})$ and hence
\[
f(R,S) = f(RS^{-1},1)S = (A+B)^{-1/2}B^{1/2}f(B^{-1/2}AB^{-1/2},1)B^{1/2}(A+B)^{-1/2}.
\]
Since $f(A,B) = (A+B)^{1/2}f(R,S)(A+B)^{1/2}$, the first assertion follows. The second assertion is now trivial. 

\medskip
Item (2): Observe that 
\[
A_{\varepsilon_1} = \begin{bmatrix} \varepsilon_1 I_{\mathcal{H}_0} & 0 \\ 0 & (A_1)_{\varepsilon_1} \end{bmatrix}, \quad 
B_{\varepsilon_2} = \begin{bmatrix} \varepsilon_2 I_{\mathcal{H}_0} & 0 \\ 0 & (B_1)_{\varepsilon_2} \end{bmatrix}, \quad 
A_{\varepsilon_1}+B_{\varepsilon_2} = \begin{bmatrix} (\varepsilon_1+\varepsilon_2) I_{\mathcal{H}_0} & 0 \\ 0 & (A_1+ B_1)_{\varepsilon_1+\varepsilon_2} \end{bmatrix}
\]
along $\mathcal{H} = \mathcal{H}_0\oplus\mathcal{H}_1$, where $(A_1)_{\varepsilon_1} := A_1 + {\varepsilon_1} I_{\mathcal{H}_1}, (B_1)_{\varepsilon_2} := B_1 + {\varepsilon_2} I_{\mathcal{H}_1}$ and $(A_1+ B_1)_{\varepsilon_1+\varepsilon_2} = A_1+B_1+(\varepsilon_1+\varepsilon_2) I_{\mathcal{H}_1} = (A_1)_{\varepsilon_1}+(B_1)_{\varepsilon_2}$.  We denote by $R_{\varepsilon_1,\varepsilon_2}, S_{\varepsilon_1,\varepsilon_2}$ the $R$ and $S$-operators associated with the pair $(A_\varepsilon,B_\varepsilon)$. Then, we have 
\begin{align*}
R_{\varepsilon_1,\varepsilon_2} &= 
\begin{bmatrix} \frac{\varepsilon_1}{\varepsilon_1+\varepsilon_2}I_{\mathcal{H}_0} & 0 \\ 0 & ((A_1)_{\varepsilon_1}+(B_1)_{\varepsilon_2})^{-1/2}(A_1)_\varepsilon ((A_1)_{\varepsilon_1}+(B_1)_{\varepsilon_2})^{-1/2} \end{bmatrix}, \\ 
S_{\varepsilon_1,\varepsilon_2} &= 
\begin{bmatrix} \frac{\varepsilon_2}{\varepsilon_1+\varepsilon_2}I_{\mathcal{H}_0} & 0 \\ 0 & ((A_1)_{\varepsilon_1}+(B_1)_{\varepsilon_2})^{-1/2}(B_1)_{\varepsilon_1}((A_1)_{\varepsilon_1}+(B_1)_{\varepsilon_2})^{-1/2} \end{bmatrix}, 
\end{align*} 
and hence
\[
f(R_{\varepsilon_1,\varepsilon_2},S_{\varepsilon_1,\varepsilon_2}) 
= 
\begin{bmatrix} 
f(\frac{\varepsilon_1}{\varepsilon_1+\varepsilon_2},\frac{\varepsilon_2}{\varepsilon_1+\varepsilon_2})I_{\mathcal{H}_0} & 0 \\ 0 & T(\varepsilon_1,\varepsilon_2) \end{bmatrix}, 
\]
where $T(\varepsilon_1,\varepsilon_2) := f(R(\varepsilon_1,\varepsilon_2),S(\varepsilon_1,\varepsilon_2))$ and 
\begin{align*}
R(\varepsilon_1,\varepsilon_2) :&= 
((A_1)_{\varepsilon_1}+(B_1)_{\varepsilon_2})^{-1/2}(A_1)_{\varepsilon_1}((A_1)_{\varepsilon_1}+(B_1)_{\varepsilon_2})^{-1/2} \\
S(\varepsilon_1,\varepsilon_2) :&=
((A_1)_{\varepsilon_1}+(B_1)_{\varepsilon_2})^{-1/2}(B_1)_{\varepsilon_1}((A_1)_{\varepsilon_1}+(B_1)_{\varepsilon_2})^{-1/2} 
\end{align*}
acting on $\mathcal{H}_1$. 

By assumption, $A_1 + B_1$ is invertible. Taking a uniform approximation to $1/\sqrt{t}$ by polynomials in $t$ on $\sigma(A_1+B_1)$ ($\subset (0,\infty)$) we see that 
\[
((A_1)_{\varepsilon_1}+(B_1)_{\varepsilon_2})^{-1/2} = (A_1+B_1+(\varepsilon_1+\varepsilon_2) I_{\mathcal{H}_1})^{-1/2} \to (A_1 + B_1)^{-1/2}
\]
in the operator norm topology as $(\varepsilon_1,\varepsilon_2)\searrow(0,0)$. Thus, $R(\varepsilon_1,\varepsilon_2) \to R$ and $S(\varepsilon_1,\varepsilon_2) \to S$ in the operator norm topology as $(\varepsilon_1,\varepsilon_2)\searrow(0,0)$. The continuity of $f(r,s)$ enables us to take its uniform approximation by polynomials in $r,s$ on $[0,1]^2$, and we see that $T(\varepsilon_1,\varepsilon_2) \to f(R,S)$ in the operator norm topology as $(\varepsilon_1,\varepsilon_2)\searrow(0,0)$. Consequently, 
\begin{align*} 
f(A_{\varepsilon_1},B_{\varepsilon_2}) 
&=
\begin{bmatrix} 
f(\varepsilon_1,\varepsilon_2)I_{\mathcal{H}_0} & 0 \\ 0 &  (A_1+B_1)^{1/2}T(\varepsilon_1,\varepsilon_2) (A_1+B_1)^{1/2} \end{bmatrix} \\
&\longrightarrow  
\begin{bmatrix} 
0 & 0 \\ 0 &  (A_1+B_1)^{1/2}f(R,S) (A_1+B_1)^{1/2} \end{bmatrix}
= f(A,B)  
 \end{align*}
in the operator norm topology as $(\varepsilon_1,\varepsilon_2)\searrow(0,0)$. 
\end{proof}

Item (1) of the above theorem explains that Pusz--Woronowicz's functional calculus gives the weighted mean $A\,\sharp_\alpha\,B = f(A,B)$ with $f(r,s)=r^{1-\alpha}s^\alpha$, $0<\alpha<1$, when $A, B$ are invertible. Also, if $f(r,s) = rs/(r+s)$ with convention $0/0 := 0$, then $f \in C([0,\infty)^2)$ and by definition, 
\begin{align*}
f(A,B) 
&= 
(A+B)^{1/2}\Big((A+B)^{-1/2} A (A+B)^{-1} B (A+B)^{-1/2}\Big)(A+B)^{1/2} \\
&= 
A(A+B)^{-1}B = (A^{-1}+B^{-1})^{-1}
\end{align*}
for any pair $A, B$ of invertible operators. (This calculation suggests that Pusz--Woronowicz's idea is similar to (but not exactly same as) Fillmore--Williams's approach to the parallel sum of Hilbert space operators \cite[\S4]{FillmoreWilliams:AdvMath71}.) Item (2) of the theorem guarantees that the $f(A,B)$ with those functions $f(r,s)$ completely agree with the corresponding operator means at least in the finite-dimensional setting. Note that those functions $f(r,s)$ were already examined by Pusz--Woronowicz \cite{PuszWoronowicz:RMP75,PuszWoronowicz:LMP78} without appealing the notion of operator means (in fact, the notion did not exist at the time).  

We will establish that Theorem \ref{T4}(2) holds in general with the strong operator convergence in place of the norm convergence. The next lemma is a key. The proof below is motivated from Fillmore--Williams's paper \cite{FillmoreWilliams:AdvMath71}.  

\begin{lemma}\label{L5} Letting  
\begin{align*}
R(\varepsilon_1,\varepsilon_2) &:= (A_1+B_1+(\varepsilon_1+\varepsilon_2) I_{\mathcal{H}_1})^{-1/2}(A_1+\varepsilon_1 I_{\mathcal{H}_1})(A_1+B_1+(\varepsilon_1+\varepsilon_2) I_{\mathcal{H}_1})^{-1/2}, \\
S(\varepsilon_1,\varepsilon_2) &:= (A_1+B_1+(\varepsilon_1+\varepsilon_2) I_{\mathcal{H}_1})^{-1/2}(B_1+\varepsilon_2 I_{\mathcal{H}_1})(A_1+B_1+(\varepsilon_1+\varepsilon_2) I_{\mathcal{H}_1})^{-1/2}  
\end{align*}
we have $ R(\varepsilon_1,\varepsilon_2) \to R$ and $S(\varepsilon_1,\varepsilon_2) \to S$ in the strong operator topology as $(\varepsilon_1,\varepsilon_2)\searrow(0,0)$. 
\end{lemma} 
\begin{proof} 
We write $I = I_{\mathcal{H}_1}$ for simplicity. We divide $R(\varepsilon_1,\varepsilon_2)$ into 
\[
\big((A_1+B_1+(\varepsilon_1+\varepsilon_2) I)^{-1/2}(A_1+\varepsilon_1 I)^{1/2}\big)\big((A_1+\varepsilon_1 I)^{1/2}(A_1+B_1+(\varepsilon_1+\varepsilon_2) I)^{-1/2}\big), 
\]
both of which are contractive by construction. Thus, it suffices to prove the convergence against a vector $y:=(A_1+B_1)^{1/2}x$ with arbitrary $x \in \mathcal{H}_1$. 

Using the spectral decomposition $A_1+B_1 = \int_0^\infty \lambda\,E(d\lambda)$ together with $\mathrm{Ker}(A_1+B_1)=\{0\}$ we have 
\[
(A_1+B_1+(\varepsilon_1+\varepsilon_2) I)^{-1/2}y = (A_1+B_1+(\varepsilon_1+\varepsilon_2) I)^{-1/2}(A_1+B_1)^{1/2}x \to x
\]
as $(\varepsilon_1,\varepsilon_2)\searrow(0,0)$, and hence 
\[
(A_1+\varepsilon_1 I)^{1/2}(A_1+B_1+(\varepsilon_1+\varepsilon_2) I)^{-1/2}y
\to A_1^{1/2}x
\]
as $(\varepsilon_1,\varepsilon_2)\searrow(0,0)$, since $(A_1+\varepsilon_1 I)^{1/2} \to A_1^{1/2}$ in the operator norm topology as $\varepsilon_1 \searrow 0$ by a uniform approximation to $\sqrt{t}$ by polynomials in $t$. 

By the famous Douglas decomposition theorem \cite{Douglas:PAMS66} (or \cite[Theorem 1.2]{FillmoreWilliams:AdvMath71}) and its proof, $X^*$ is given by $(A_1+B_1)^{-1/2}A_1^{1/2}$ (as a composition).  This means that the range $\mathrm{ran}(A_1^{1/2})$ sits in the domain $\mathrm{dom}((A_1+B_1)^{-1/2})$. With this observation, we have
\begin{align*}
&\Vert(A_1+B_1+(\varepsilon_1+\varepsilon_2) I)^{-1/2}(A_1+\varepsilon_1 I)^{1/2}A_1^{1/2}x - X^* A_1^{1/2} x\Vert_{\mathcal{H}_1}  \\
&\leq 
\Vert \big((A_1+B_1+(\varepsilon_1+\varepsilon_2) I)^{-1/2} A_1^{1/2}\big)\big((A_1+\varepsilon_1 I)^{1/2}x-A_1^{1/2}x\big)\Vert_{\mathcal{H}_1} \\
&\qquad\qquad+ 
\Vert(A_1+B_1 + (\varepsilon_1+\varepsilon_2) I)^{-1/2} A_1 x - (A_1+B_1)^{-1/2} A_1 x\Vert_{\mathcal{H}_1} \\
&\leq 
\Vert (A_1+\varepsilon_1 I)^{1/2}x-A_1^{1/2}x\Vert_{\mathcal{H}_1} + 
\Vert((A_1+B_1 + (\varepsilon_1+\varepsilon_2) I)^{-1/2} - (A_1+B_1)^{-1/2}) A_1 x\Vert_{\mathcal{H}_1}
\end{align*}
since $A_1 \leq A_1+B_1+(\varepsilon_1+\varepsilon_2) I$ implies $\Vert (A_1+B_1+(\varepsilon_1+\varepsilon_2) I)^{-1/2} A_1^{1/2}\Vert \leq 1$ by the Douglas decomposition theorem again. The first term on the last line converges to $0$ as remarked before. The second term on the last line also converges to $0$ as follows. 

Since $\mathrm{ran}(A_1) \subset \mathrm{ran}(A_1^{1/2}) \subset \mathrm{dom}((A_1+B_1)^{-1/2})$, we observe that $\int_0^\infty \frac{1}{\lambda} \Vert E(d\lambda) A_1x\Vert_{\mathcal{H}_1}^2 < +\infty$. Thus, the dominated convergence theorem shows that 
\[
\Vert((A_1+B_1 + (\varepsilon_1+\varepsilon_2) I)^{-1/2} - (A_1+B_1)^{-1/2})A_1 x\Vert_{\mathcal{H}_1}^2 = 
\int_0^\infty \Big|\frac{1}{\sqrt{\lambda + (\varepsilon_1+\varepsilon_2)}} - \frac{1}{\sqrt{\lambda}}\Big|^2\,\Vert E(d\lambda) A_1x\Vert_{\mathcal{H}_1}^2
\]
converges to $0$ as $(\varepsilon_1,\varepsilon_2)\searrow(0,0)$, since $\mathrm{Ker}(A_1+B_1)=\{0\}$ and $(A_1+B_1)^{-1/2} = \int_0^\infty \lambda^{-1/2}\,E(d\lambda)$. 

Consequently, we obtain that 
\[
(A_1+B_1+(\varepsilon_1+\varepsilon_2) I)^{-1/2}(A_1+\varepsilon_1 I)^{1/2}A_1^{1/2}x \to X^* A_1^{1/2} x = X^* X (A_1+B_1)^{1/2} x = Ry 
\]
as $(\varepsilon_1,\varepsilon_2)\searrow(0,0)$. Summing up the discussion so far, we have $R(\varepsilon_1,\varepsilon_2)\to R$ as $(\varepsilon_1,\varepsilon_2)\searrow0$. Then,  $S(\varepsilon_1,\varepsilon_2) = I-R(\varepsilon_1,\varepsilon_2) \to I-R = S$ in the strong operator topology as $(\varepsilon_1,\varepsilon_2)\searrow(0,0)$ too. 
\end{proof} 

The above lemma enables us to prove the following theorem: 

\begin{theorem}\label{T6} For any homogeneous $f \in C([0,\infty)^2)$, we have 
\begin{align*}
f(A,B) 
= \lim_{(\varepsilon_1,\varepsilon_2)\searrow(0,0)} f(A_{\varepsilon_1},B_{\varepsilon_2}) 
\end{align*}
in the strong operator topology, where $A_{\varepsilon_1} := A+\varepsilon_1 I_\mathcal{H}, B_{\varepsilon_2} := B+\varepsilon_2 I_\mathcal{H} $. 
\end{theorem}
\begin{proof} 
We keep the notations in Lemma \ref{L5} and its proof. As in the proof of Theorem \ref{T4}(2) it suffices to prove that $f(R(\varepsilon_1,\varepsilon_2),S(\varepsilon_1,\varepsilon_2)) \to f(R,S)$ in the strong operator topology  as $(\varepsilon_1,\varepsilon_2)\searrow(0,0)$. This indeed follows from Lemma \ref{L5} with a uniform approximation to $f(r,s)$ by polynomials in $r,s$ on $[0,1]^2$.    
\end{proof} 

Remark that the above theorem shows that Pusz--Woronowicz's functional calculus $f(A,B)$ with $f(r,s) = r^{1-\alpha}s^\alpha$ and $f(r,s) = rs/(r+s)$ completely agree with the weighted mean $A\,\sharp_\alpha\,B$ and the parallel sum $A:B$, respectively. Moreover, any operator connection $\sigma : B(\mathcal{H})^+\times B(\mathcal{H})^+ \to B(\mathcal{H})^+$ is captured in terms of Pusz--Woronowicz's functional calculus $f(A,B)$ with $f(r,s) I_\mathcal{H} := (r I_\mathcal{H})\sigma(s I_\mathcal{H})$. This last fact was essentially obtained in \cite{Fujii:MathJapon88} utilizing \cite[Theorem 3.4]{KuboAndo:MathAnn80}; in fact, the approach to operator means there is nothing less than an adaptation of Pusz--Woronowicz's functional calculus to Kubo--Ando's operator means from the viewpoint of this note. 

\medskip
Here is a technical remark, which is quite a natural formula. (More general formulas involving `sections' appeared in the proof of the main theorem of Pusz--Woronowicz's second paper \cite{PuszWoronowicz:LMP78}.)  

\begin{remark}\label{R8} For any homogeneous $f \in \mathscr{B}_\mathrm{locb}([0,\infty)^2)$ and $a, b > 0$ we define $g(r,s) := f(a^{-1}r,b^{-1}s)$, which falls into $\mathscr{B}_\mathrm{locb}([0,\infty^2)$ and is homogeneous again. Then, $f(A,B) = g(aA,bB)$ holds.  
\end{remark}

When $f$ is in $C([0,\infty)^2)$, one can easily confirm this remark by using Theorem \ref{T6}. Then one applies the trick utilizing the monotone class theorem like \cite[Theorem 1.2]{PuszWoronowicz:RMP75} (or Lemma \ref{L3} in this note). Here one remarks that there is a one-to-one correspondence $f(r,s) \leftrightarrow \psi(t)$ between all the homogeneous functions in $\mathscr{B}_\mathrm{locb}([0,\infty)^2)$ and $\mathscr{B}_\mathrm{b}([0,1])$ by $\psi(t) = f(t,1-t)$ and $f(r,s) = (r+s)\psi(r/(r+s))$ with $f(0,0) = 0$. Thus one should use the monotone class theorem on $\mathscr{B}_\mathrm{b}([0,1])$ rather than $\mathscr{B}_\mathrm{locb}([0,\infty)^2)$. The correspondence will explicitly be used in our explanation on Pusz--Woronowicz's Wigner--Yanase--Dyson--Lieb (WYDL in short) type theorem below.  

\section{Relations to recent studies} 

\subsection{Operator perspective}

We will clarify that the joint operator convexity of operator perspectives could be thought of as a simple consequence from Pusz--Woronowicz's work \cite{PuszWoronowicz:LMP78} if one already knew Theorem \ref{T4}. Utilizing their work has been missing in the development of operator means and perspectives. Pusz--Woronowicz's method dealing with the joint operator convexity is quite interesting. But, it seems that only a few specialists on operator means and perspectives realize it well. So, the purpose here is to explain it in terms of Hilbert space bounded operators.  

\medskip
By the construction of the pair $(R,S)$ from a given one $(A,B)$ together with the regularity of usual functional calculus $(R,S) \mapsto f(R,S)$, it is rather easy to see that Pusz--Woronowicz's functional calculus $f(A,B)$ is an example of regular map in the sense of \cite[Definition 2.1]{EffrosHansen:AFA14}. Here, a regular map means a mapping $(A,B) \mapsto F(A,B)$ from a convex subset of pairs of positive operators on a Hilbert space $\mathcal{H}$ to $B(\mathcal{H})$ that are compatible with unitary equivalence and direct sum like usual functional calculus; namely, $F(UAU^*,UBU^*) = UF(A,B)U^*$ for any unitary operator $U$, and if  $A = A_1\oplus A_2, B=B_1\oplus B_2$ on a direct sum of Hilbert spaces $\mathcal{H} = \mathcal{H}_1\oplus\mathcal{H}_2$., then $F(A,B) = F(A_1,B_1)\oplus F(A_2,B_2)$ holds along $\mathcal{H} = \mathcal{H}_1\oplus\mathcal{H}_2$. 

The next proposition is almost trivial by the proof of Theorem \ref{T4}(2). It is exactly item (iii) in \cite[Theorem 2.2]{EffrosHansen:AFA14}. We give its statement in a bit wider setup than before for a wider applicability.  

\begin{proposition}\label{P8} If $A$ is invertible and if $B_n \to B$ in the strong operator topology {\rm(}or operator norm{\rm)} as $n\to\infty$, then $f(A,B_n) \to f(A,B)$ in the strong operator topology {\rm(}resp. operator norm{\rm)} as $n\to\infty$ for any homogeneous $f \in \mathscr{B}([0,\infty)^2)$ that is continuous except $r=0$. {\rm(}Strictly speaking, $f(A,B)$ is not defined for such a function $f$ in \S3.1, but the construction clearly works well for the pairs considered here.{\rm)} The same holds when the roles of $A,B$ are interchanged. 
\end{proposition} 
\begin{proof}
Since $A$ is invertible, we have $A \geq \delta I$ for some $\delta > 0$, and hence $A + B_n \geq \delta I$ for all $n$. Therefore, all the $\sigma(A+B_n)$ and $\sigma(A)$ sit in $[\delta,\gamma]$ for some $\gamma > \delta$. Hence, taking a uniform approximation to $x^{-1/2}$ by polynomials over $[\delta,\gamma]$ we have $(A+B_n)^{-1/2} \to (A+B)^{-1/2}$ in the strong operator topology as $n\to\infty$. Moreover, $(A+B_n)^{-1/2} A (A+B_n)^{-1/2} \geq \delta(A+B_n)^{-1} \geq (\delta/\gamma)I > 0$. These explain that the proof of Theorem \ref{T6} does work for showing this proposition.  
\end{proof}

Pusz--Woronowicz's functional calculus $f(A,B)$ is homogeneous, i.e., $f(\lambda A,\lambda B)=\lambda f(A,B)$ for all $\lambda\geq0$, and hence satisfies items (i) in \cite[Theorem 2.2]{EffrosHansen:AFA14}. In fact, this identity trivially holds as $0=0$ when $\lambda=0$. When $\lambda>0$, one can easily check that the $R$ and the $S$-operators are not changed under $(A,B) \mapsto (\lambda A,\lambda B)$, and this fact immediately implies that the desired identity holds even in this case. A stronger homogeneity will be given later; see Remark \ref{R10}. 

\medskip
We then examine when Pusz--Woronowicz's functional calculus $f(A,B)$ is convex as a function in two variables. A complete solution to the question of convexity was given by Pusz--Woronowicz \cite{PuszWoronowicz:LMP78} in their formalism. We will give its translation into the Hilbert space operator formalism with a proof in a restricted setup. Here, we point out that Pusz--Woronowicz dealt with more general functions $f(r,s)$ possibly taking $+\infty$, \emph{but such a perfect treatment in the level of operators needs the use of unbounded operators}. On the other hand, we believe that the essence of their idea `faithfully' appears in the explanation below. 

\begin{theorem}\label{T9} {\rm(Pusz--Woronowicz's WYDL type theorem \cite{PuszWoronowicz:LMP78})} For a homogeneous and real-valued $f \in \mathscr{B}_\mathrm{locb}([0,\infty)^2)$, the following are equivalent{\rm:}
\begin{itemize}
\item[(i)] $f(V^*AV,V^*BV) \leq V^* f(A,B)V$ holds for any pair $A,B$ of positive bounded operators on a Hilbert space $\mathcal{H}$ and any bounded linear map $V$ from another Hilbert space $\mathcal{K}$ to $\mathcal{H}$. 
\item[(ii)] $t \in [0,1] \mapsto f(t,1-t) \in \mathbb{R}$ is operator convex. 
\item[(ii')] $t \in [0,1] \mapsto f(1-t,t) \in \mathbb{R}$ is operator convex. 
\end{itemize}
Moreover, if $f(r,s)$ is in $C([0,\infty)^2)$, then the following conditions are also equivalent to item {\rm(i)}{\rm:} 
\begin{itemize}
\item[(iii)] $t \in [0,\infty) \mapsto f(t,1) \in \mathbb{R}$ is operator convex. 
\item[(iii')] $t  \in [0,\infty) \mapsto f(1,t) \in \mathbb{R}$ is operator convex. 
\end{itemize}
\end{theorem}
\begin{proof} We first recall an equivalent definition of operator convexity. A real-valued function $h(t)$ on an interval $I$ is \emph{operator convex} if $h(W^* X W) \leq W^* h(X)W$ holds for any bounded self-adjoint operator $X$ on a Hilbert space $\mathcal{L}_1$ with $\sigma(X) \subseteq I$ and any isometry $W$ from another Hilbert space $\mathcal{L}_2$ to $\mathcal{L}_1$. See e.g., \cite[Theorem 2.5.7(ii)]{Hiai:LectNotes10}. 

\medskip
(i) $\Rightarrow$ (ii), (ii'): Trivial, because $f(T,I_\mathcal{H}-T)$ (or $f(I_\mathcal{H}-T,T)$) can also be understood as the functional calculus of $T$ with $t \mapsto f(t,1-t)$ (resp.\ $f(1-t,t)$); see the proofs of Proposition \ref{P2} and Theorem \ref{T4}.  

\medskip
(ii) $\Rightarrow$ (i), (ii') $\Rightarrow$ (i): By symmetry, we will prove only the former. Denote by $R, R'$ the $R$-operators associated with the pairs $(A,B), (V^* AV, V^*BV)$, respectively. 

Observe that $\Vert (V^*AV+V^*B V)^{1/2} x\Vert_\mathcal{K} = \Vert (A+B)^{1/2} Vx\Vert_\mathcal{H}$ for any $x \in \mathcal{K}$. Hence, we have an isometry $U : \mathcal{K}_1 \to \mathcal{H}_1$ sending $(V^*AV+V^*B V)^{1/2} x$ to $(A+B)^{1/2} Vx$ with arbitrary $x \in \mathcal{K}$, where $\mathcal{K} = \mathcal{K}_0\oplus\mathcal{K}_1$ denotes the decomposition associated with the pair $(V^* AV, V^*BV)$. This implies that $U(V^*AV+V^*BV)^{1/2} = (A+B)^{1/2}V$. 

For any $x, y \in \mathcal{K}$ we have 
\begin{align*} 
&(R'(V^*AV+V^*B V)^{1/2} x\,|\,(V^*AV+V^*B V)^{1/2}y)_{\mathcal{K}_1} \\
&= 
((V^*AV)^{1/2}x\,|\,(V^*AV)^{1/2}y)_\mathcal{K} 
= 
(A^{1/2}Vx\,|\,A^{1/2}Vy)_{\mathcal{H}} \\
&= 
(R(A+B)^{1/2}Vx\,|\,(A+B)^{1/2}Vy)_{\mathcal{H}_1} \\
&= 
(U^*RU(V^*AV+V^*B V)^{1/2} x\,|\,(V^*AV+V^*B V)^{1/2}Vy)_{\mathcal{K}_1}, 
\end{align*}
and hence $R' = U^* RU$ and $S'=I_{\mathcal{K}_1}-R' = U^*(I_{\mathcal{H}_1}-R)U = U^* SU$ (since $U^*U=I_{\mathcal{K}_1}$). By item (ii), it follows that 
\[
f(R',S')=g(R') = g(U^* R U) \leq U^* g(R)U = U^* f(R,S)U 
\]
with $g(t) := f(t,1-t)$. This inequality and $U(V^*AV+V^*BV)^{1/2} = (A+B)^{1/2}V$ imply item (i).  

\medskip
(i) $\Rightarrow$ (iii), (iii'): Trivial, because $f(T,I_\mathcal{H})$ (or $f(I_\mathcal{H},T)$) can also be understood as the functional calculus of $T$ with $t \mapsto f(t,1)$ (resp.\ $f(1,t)$) as in (i) $\Rightarrow$ (ii), (ii'). 

\medskip
(iii) $\Rightarrow$ (ii), (iii') $\Rightarrow$ (ii'): By symmetry, we will prove only the former. Assume item (iii), that is, $\psi(t) := f(t,1)$ is operator convex on $[0,\infty)$. For each $0<\delta<1$ we set $c_\delta:=(1-\delta)/\delta$ and define $g_\delta(r,s):=((r+s)/c_\delta)\psi(c_\delta r/(r+s))$ with $g_\delta(0,0)=0$. Then $g_\delta(t,c_\delta-t)=\psi(t)$ for $0\leq t \leq c_\delta$ and $g_\delta(t,1-t)=\psi(c_\delta t)/c_\delta$ is operator convex on $0\leq t\leq 1$. By what we have shown above (i.e., (i) $\Leftrightarrow$ (ii)), we observe that $(A,B) \mapsto g_\delta(A,B)$ satisfies the property of item (i). 

Observe that $f(t,1-t) = g_\delta(t,c_\delta(1-t)-t)$ for every $0\leq t\leq 1-\delta$. Choose an arbitrary $T$ with $\sigma(T) \subseteq [0,1-\delta]$ and an arbitrary isometry $V$. Then $\sigma(V^* TV) \subseteq [0,1-\delta]$ and 
\[
g_\delta(V^*TV,c_\delta(I_\mathcal{H}-V^*TV)-V^* TV) = 
g_\delta(V^*TV,V^*(c_\delta(I_\mathcal{H}-T)-T)V)
\leq 
V^* g_\delta(T,c_\delta(I_\mathcal{H}-T)-T)V
\]
by the property of item (i). Consequently, $t \mapsto f(t,1-t)$ must be operator convex on $[0,1-\delta]$. Since $0<\delta<1$ is arbitrary, it follows, by the continuity of $f(r,s)$, that item (ii) holds on the whole $[0,1]$.
\end{proof} 

\begin{remark}\label{R10} If the closure of the range of $V$ contains $\mathcal{H}_1$, then the isometry $U$ in the proof of {\rm(ii)} $\Rightarrow$ {\rm(i)} is actually a unitary transform, and hence $f(R',S') = U^*f(R,S)U$. This explains that the `operator homogeneity' $f(V^* AV,V^* BV) = V^*f(A,B)V$ holds under this assumption on $V$ without any extra ones {\rm(}cf.\ \cite[Theorem 3]{Fujii:MathJapon88} which dealt with only operator means{\rm)}. This remark explains that Pusz--Woronowicz's functional calculus is a natural noncommutative one, and makes an abstract approach to it possible. This aspect will be discussed elsewhere in a more general setup allowing unbounded functions.  
\end{remark}

The proof of Theorem \ref{T9} clearly gives the next variant. In fact, this variant is easier to prove now. 

\begin{remark}\label{R11} For a homogeneous, real-valued $f \in \mathscr{B}_\mathrm{locb}((0,\infty)^2)$ the following are equivalent{\rm:}
\begin{itemize}
\item[(i)] $f(V^*AV,V^*BV) \leq V^* f(A,B)V$ holds for any pair $A,B$ of positive invertible bounded operators on a Hilbert space $\mathcal{H}$ and any isometry\,{\rm(!)} $V$ from another Hilbert space $\mathcal{K}$ to $\mathcal{H}$. 
\item[(ii)] $t \in (0,1) \mapsto f(t,1-t) \in \mathbb{R}$ is operator convex. 
\item[(ii')] $t \in (0,1) \mapsto f(1-t,t) \in \mathbb{R}$ is operator convex. 
\item[(iii)] $t \in (0,\infty) \mapsto f(t,1) \in \mathbb{R}$ is operator convex. 
\item[(iii')] $t \in (0,\infty) \mapsto f(1,t) \in \mathbb{R}$ is operator convex. 
\end{itemize}
\end{remark}

The facts that we have explained so far say that if one starts with a homogeneous real-valued $f \in \mathscr{B}_\mathrm{locb}([0,\infty)^2)$ satisfying the equivalent conditions in Theorem \ref{T9}, then $(A,B) \mapsto f(A,B)$ becomes a non-commutative perspective function in the sense of Effros--Hansen \cite[Theorem 2.2]{EffrosHansen:AFA14}. Since Effros--Hansen's notion of non-commutative perspective functions cannot be used to capture important examples like operator relative entropy (due to the choice of their domains), a more important observation here is the following: One may think that Pusz--Woronowicz's work \cite{PuszWoronowicz:LMP78} implicitly contains a proof to a consequence of recent works \cite{Effros:NASUSA09,EbadianNikoufarGordji:PNAS11} modulo Theorem \ref{T4}. Remark \ref{R11} says, without appealing to those works, that if one starts with an operator convex function $g(t)$ on $(0,\infty)$ (not $[0,+\infty)$), then $(A,B) \mapsto f(A,B)$ with $f(r,s) := sg(r/s)$ on $(0,\infty)^2$ defined on the pairs of positive invertible operators enjoys the joint operator convexity, and the joint operator convexity characterizes the operator convexity of $g(t)$. Here we recall that a (widely accepted) definition of operator perspective $P_g(A,B)$ associated with an operator convex function $g(t)$ on $(0,+\infty)$ is $P_g(A,B) = B^{1/2}g(B^{-1/2}AB^{-1/2})B^{1/2}$ for positive \emph{invertible} operators $A,B$ on a Hilbert space, and $P_g(A,B) = f(A,B)$ with $f(r,s) := sg(r/s)$ holds thanks to Theorem \ref{T4}. Note that the notion of operator perspectives has not yet been established well for general positive (not necessarily invertible) operators on Hilbert spaces. 

\medskip
We point out that Pusz--Woronowicz proved, as a corollary of their WYDL type theorem, that \emph{the operator convexity {\rm(}or concavity{\rm)} is preseved under linear fractional transformations}, which seems not well known at least explicitly. This is out of scope of the purpose here. See \cite[Corollary 2]{PuszWoronowicz:LMP78} for details. 

\subsection{Quantum information theoretic quantities} 

Pusz--Woronowicz's functional calculus seems useful for the study of quantum information theoretic quantities, but the class $\mathscr{B}_\mathrm{locb}([0,\infty)^2)$ is too small to apply it as already pointed out by Pusz--Woronowicz \cite{PuszWoronowicz:LMP78}; in fact, the important example $f(r,s) = r\log(r/s)$ takes $+\infty$ on the line of $s=0$. However, one can still consider $f(A,B)$ for such a function as a kind of `generalized positive operator' affiliated with $B(\mathcal{H})$ in the sense of Haagerup \cite[\S1]{Haagerup:JFA79} that fits the notion of `generalized quadratic forms' in the sense of \cite{PuszWoronowicz:LMP78}. Such a delicate study needs many extra preparations; we will give here only a bit heuristic explanation about it in the finite-dimensional setting for the reader's convenience. For quantum information quantities, we refer the reader to \cite{HiaiMosonyi:RMP17}. 

Let $A, B$ be positive operators on a finite-dimensional Hilbert space $\mathcal{H}$. Let $\mathcal{H}=\mathcal{H}_0\oplus\mathcal{H}_1$ and $R, S$ be associated with the pair $(A,B)$ as in \S3. Let $R=\sum_{\lambda \in \sigma(R)} \lambda\,P_\lambda$ and $S=\sum_{\mu \in \sigma(S)} \mu\,Q_\mu$ be the spectral decomposition. Then, we have
\[
f(A,B) = \sum_{\lambda\in\sigma(R)} f(\lambda,1-\lambda)\,(A+B)^{1/2}P_\lambda (A+B)^{1/2} = \sum_{\mu\in\sigma(S)} f(1-\mu,\mu)\,(A+B)^{1/2}Q_\mu (A+B)^{1/2},
\]
\emph{but we allow some $f(\lambda,1-\lambda)$ or $f(1-\mu,\mu)$ to be $+\infty$}. This is exactly the same formula as \cite[Theorem 4.2]{Hiai:JMathPhys19}. (We thank Hiai for letting us be aware of it.) In particular, with $h(r,s) := r\log(r/s), f_\alpha(r,s) := r^{1-\alpha} s^\alpha$ ($0 < \alpha \leq 2$), we observe that $\mathrm{Tr}(h(A,B))=+\infty$ if and only if $0 \in \sigma(S)$, that is, $\mathrm{Ker}(B)\neq\mathrm{Ker}(A+B)$, and moreover, that $\mathrm{Tr}(f_\alpha(A,B))=+\infty$ if and only if $\alpha>1$ and $0 \in \sigma(R)$, that is, $\mathrm{Ker}(A)\neq\mathrm{Ker}(A+B)$. On the other hand, Theorem \ref{T4}(1) shows that  
\begin{align*}
\mathrm{Tr}(h(A,B))&=\mathrm{Tr}(A\log(A^{1/2}B^{-1}A^{1/2})) 
= \mathrm{Tr}(B^{1/2}AB^{-1/2}\log(B^{-1/2}AB^{-1/2})), \\
\mathrm{Tr}(f_\alpha(A,B))&=\mathrm{Tr}(B(B^{-1/2}AB^{-1/2})^{1-\alpha}) 
= \mathrm{Tr}(A(A^{-1/2}BA^{-1/2})^\alpha)\ (= \mathrm{Tr}(A\,\sharp_\alpha\,B)\ \text{if $\alpha <1$}),
\end{align*}
when both $A, B$ are invertible. 

The operator $R$ is sufficiently explicit, that is, $R=(A+B)^{-1/2}A(A+B)^{-1/2}$ holds, but $(A+B)^{-1/2}$ should be understood as `the partial inverse' on $\mathrm{Ker}(A+B)^\perp$ in the finite-dimensional setting ({\it n.b.}, this understanding is essentially valid even in the infinite-dimensional setting; see the proof of Lemma \ref{L5}). 

Pusz--Woronowicz gave a general method of obtaining a variational expression of $f(A,B)$ in terms of positive forms. In fact, they gave explicit formulas in the cases of $h(r,s)$ and $f_\alpha(r,s)$ ($0 < \alpha < 1$) in \cite[\S2]{PuszWoronowicz:LMP78}. The method is essentially based on the well-known variational expression of the parallel sum $A:B$ (see e.g.\ \cite[Lemma 3.1.5]{Hiai:LectNotes10}), which is nothing but \cite[Eq(1.2)]{PuszWoronowicz:RMP75}. We recommend the reader to consult \cite[Appendix]{PuszWoronowicz:RMP75},\cite[\S2]{PuszWoronowicz:LMP78} for the details as well as \cite[\S4]{Donald:CMP86} due to Donald as a detailed exposiosion in the case of $h(r,s) = r\log(r/s)$. 

\section*{Acknowledgements} 
The second author thanks Shigeru Yamagami and Fumio Hiai for fruitful conversations. In fact, a casual conversation with Yamagami a few years ago led him to take a serious look at Pusz--Woronowicz's works and to suggest the first author to study them together after her reading of \cite{Hiai:LectNotes10}. He also benefited from a conversation with Hiai at a conference in Dec.\ 2019. 

Finally, both the authors thank Fumio Hiai for several comments to this note, and also thank Mitsuru Uchiyama and the referee for pointing out several typos and for giving comments on the presentation.

\end{document}